\documentclass[12pt]{amsart}

\makeatletter
  \CheckCommand*\refstepcounter[1]{\stepcounter{#1}%
      \protected@edef\@currentlabel
       {\csname p@#1\endcsname\csname the#1\endcsname}%
  }
  \renewcommand*\refstepcounter[1]{\stepcounter{#1}%
    \protected@edef\@currentlabel
      {\csname p@#1\expandafter\endcsname\csname the#1\endcsname}%
  }
  \def\labelformat#1{\expandafter\def\csname p@#1\endcsname##1}
  \DeclareRobustCommand\Ref[1]{\protected@edef\@tempa{\ref{#1}}%
     \expandafter\MakeUppercase\@tempa
  }
  \makeatother

  \makeatletter
  \newcommand{\numberlike}[2]{%
     \expandafter\def\csname c@#1\endcsname{%
         \expandafter\csname c@#2\endcsname}%
  }
  \makeatother

  \def\DefaultNumberTheoremWithin{section}
  \labelformat{section}{Section~#1}
\labelformat{subsection}{Section~#1}
 \theoremstyle{plain}
  \newtheorem{lemma}{Lemma}
     \numberwithin{lemma}{\DefaultNumberTheoremWithin}
     \labelformat{lemma}{Lemma~#1}
  
     \numberwithin{claim}{\DefaultNumberTheoremWithin}
     \numberlike{claim}{lemma}
     \labelformat{claim}{Claim~#1}
  \newtheorem{theorem}{Theorem}
     \numberwithin{theorem}{\DefaultNumberTheoremWithin}
     \numberlike{theorem}{lemma}
     \labelformat{theorem}{Theorem~#1}
  
     \numberwithin{corollary}{\DefaultNumberTheoremWithin}
     \numberlike{corollary}{lemma}
     \labelformat{corollary}{Corollary~#1}
  
     \numberwithin{proposition}{\DefaultNumberTheoremWithin}
     \numberlike{proposition}{lemma}
     \labelformat{proposition}{Proposition~#1}
  
     \numberwithin{conjecture}{\DefaultNumberTheoremWithin}
     \numberlike{conjecture}{lemma}
     \labelformat{conjecture}{Conjecture~#1}

  \theoremstyle{definition}
  
     \numberwithin{definition}{\DefaultNumberTheoremWithin}
     \numberlike{definition}{lemma}
     \labelformat{definition}{Definition~#1}

  \theoremstyle{definition}
  
     \numberwithin{question}{\DefaultNumberTheoremWithin}
     \numberlike{question}{lemma}
     \labelformat{question}{Question~#1}

  \theoremstyle{definition}
  
     \numberwithin{problem}{\DefaultNumberTheoremWithin}
     \numberlike{problem}{lemma}
     \labelformat{problem}{Problem~#1}

 \theoremstyle{definition}
 
 	\numberwithin{notation}{\DefaultNumberTheoremWithin}
 	\numberlike{notation}{lemma}
 	\labelformat{notation}{Notation~#1}
 	
  \theoremstyle{remark}
  \newtheorem{remark}{Remark}
     \numberwithin{remark}{\DefaultNumberTheoremWithin}
     \numberlike{remark}{lemma}
     \labelformat{remark}{Remark~#1}
  \theoremstyle{remark}

     \numberwithin{example}{\DefaultNumberTheoremWithin}
     \numberlike{example}{lemma}
     \labelformat{example}{Example~#1}
  
     \labelformat{case}{Case~#1}
     \numberwithin{case}{lemma}
  
     \labelformat{step}{Step~#1}
     \numberwithin{step}{lemma}

\newcommand{\ZZ}{ \ensuremath{\mathbb{Z}}}

\newcommand{\init}{\ensuremath{\mathrm{in}}\hspace{1pt}}

\newcommand{\rev}{{\mathrm{{rev}}}}

\newcommand{\Hilb}{\mathrm{Hilb}}

\def\cocoa{{\hbox{\rm C\kern-.13em o\kern-.07em C\kern-.13em o\kern-.15em A}}}

\newcommand{\rr}{{\langle r \rangle}}

\begin{document}

\title
{Lefschetz properties and the Veronese construction}

\author{Martina Kubitzke}
\thanks{Research of the first author was supported by the Austrian Science Foundation (FWF) through grant Y463-N13.
Research of the second author was supported by KAKENHI 22740018.}
\address{
Martina Kubitzke,
Fakult\"at f\"ur Mathematik, 
Universit\"at Wien, 
Garnisongasse 3, 
1090 Wien,
Austria.
}
\email{martina.kubitke@univie.ac.at}

\author{Satoshi Murai}
\address{
Satoshi Murai,
Department of Mathematical Science,
Faculty of Science,
Yamaguchi University,
1677-1 Yoshida, Yamaguchi 753-8512, Japan.
}
\email{murai@yamaguchi-u.ac.jp}


\begin{abstract}
In this paper,
we investigate Lefschetz properties of Veronese subalgebras.
We show that, for a sufficiently large $r$,
the $r$\textsuperscript{th} Veronese subalgebra of a Cohen-Macaulay standard graded $K$-algebra
has properties similar to the weak and strong Lefschetz properties,
which we call the `almost weak' and `almost strong' Lefschetz properties.
By using this result,
we obtain new results on $h$- and $g$-polynomials of Veronese subalgebras.
\end{abstract}

\maketitle

\section{Introduction}

Let $K$ be a field of characteristic $0$.
For a standard graded (commutative) $K$-algebra $A= \bigoplus_{i \geq 0} A_i$
and for an integer $r \geq 1$,
the $K$-algebra $A^\rr:=\bigoplus_{i \geq 0} A_{ir}$ is called the {\em $r$\textsuperscript{th} Veronese subalgebra} of $A$.
Quite recently, it has been of interest to study $h$-polynomials of Veronese subalgebras \cite{BS,BW,KW}.
In this paper,
we investigate Lefschetz properties of Veronese subalgebras of Cohen-Macaulay standard graded $K$-algebras
and obtain new results on $h$- and $g$-polynomials of Veronese subalgebras.

We first recall some basics on Hilbert series and $h$-polynomials.
The {\em Hilbert series} of a standard graded $K$-algebra $A=\bigoplus_{i \geq 0} A_i$ is the formal power series
$\Hilb(A,t):= \sum_{i \geq 0} (\dim_K A_i) t^i$.
It is known that $\Hilb(A,t)$ is a rational function of the form
$\Hilb(A,t)=(h_0+h_1t+ \cdots + h_st^s)/(1-t)^d$,
where each $h_i$ is an integer and where $d=\dim A$ is the Krull dimension of $A$
(see e.\,g., \cite[\S 4.1]{BH}).
The polynomial 
\begin{equation*}
h_A(t):=h_0+h_1t+ \cdots + h_st^s
\end{equation*}
and the polynomial 
\begin{equation*}
g_A(t):=h_0+(h_1-h_0)t+ \cdots + (h_{\lfloor \frac s 2 \rfloor}-h_{\lfloor \frac s 2 \rfloor-1})t^{\lfloor \frac s 2 \rfloor}
\end{equation*}
are called the {\em $h$-polynomial} of $A$ and the {\em $g$-polynomial} of $A$, respectively. 
Here, $\lfloor x \rfloor$ denotes the integer part of $x$.

For $h$-polynomials of Veronese subalgebras,
Brenti and Welker \cite[Corollary 1.6]{BW} proved that if $h_A(t) \in \ZZ_{\geq 0}[t]$, then,
for sufficiently large $r$,
the polynomial $h_{A^\rr}(t)$ has only real zeros,
and in particular, the coefficient sequence of $h_{A^\rr}(t)$
is unimodal and log-concave.
Moreover, it was proved in \cite{KW} that if $h_A(t) \in \ZZ_{\geq 0}[t]$
and if $r \geq \max\{\dim A, \deg h_A(t)\}$,
then the $g$-polynomial of $A^\rr$ is the $f$-polynomial of a simplicial complex.
Algebraically,
the unimodality of the $h$-polynomial of a graded $K$-algebra
is closely related to Lefschetz properties of Artinian graded $K$-algebras.
One of the main purposes of this paper is to find a connection between Lefschetz properties
and the Veronese construction.

We first consider the $k$-Lefschetz property and almost strong Lefschetz property, introduced in \cite{KN}.
Let $A=\bigoplus_{i=0}^s A_i$ be a standard graded Artinian $K$-algebra, where $\dim_K A_s >0$.
For an integer $k \geq 1$,
we say that $A$ has the {\em $k$-Lefschetz property} if there is a linear form $w \in A_1$ such that
the multiplication $w^{k-2i} : A_i \to A_{k-i}: \; p\mapsto w^{k-2i}p$ is injective for $0\leq i \leq \lfloor\frac{k-1}{2}\rfloor$. The linear form $w$ is referred to as a {\em $k$-Lefschetz element} for $A$. 
If $A$ has the $(s-1)$-Lefschetz property, then we call it {\em almost strong Lefschetz}. 
An important consequence of the almost strong Lefschetz property
is that  if $A$ is almost strong Lefschetz, then the multiplication 
$w : A_i \to A_{i+1}$ is injective for $0\leq i \leq \lfloor \frac{s}{2}\rfloor -1$
and the coefficient sequence of $g_A(t)$ becomes an $M$-sequence,
namely, there is a standard graded Artinian $K$-algebra $B$ such that $g_A(t)=\Hilb(B,t)$.

We use Lefschetz properties to study $h$-polynomials of $A^\rr$ in the following way.
Let $A$ be a Cohen-Macaulay standard graded $K$-algebra of dimension $d$.
A {\em linear system of parameters} (l.s.o.p.\ for short) for $A$ is a sequence $\Theta=\theta_1,\dots,\theta_d \in A_1$
such that $\dim_K A/\Theta A < \infty$.
Note that an l.s.o.p.\ for $A$ exists if $K$ is infinite, see e.g., \cite{St-CCA}.
For an l.s.o.p.\ $\Theta=\theta_1,\dots,\theta_d$ for $A$ and for an integer $r \geq 1$,
we write
\begin{equation*}
A_\Theta^\rr:=A^\rr/(\theta_1^r A^\rr + \cdots + \theta_d^r A^\rr).
\end{equation*}
We will see in Section 2, that  
$\theta_1^r,\dots,\theta_d^r$ is an l.s.o.p.\ for $A^\rr$, and that the Hilbert series of $A_\Theta^\rr$
is equal to the $h$-polynomial of $A^\rr$. As a consequence, the $h$-polynomial of $A^\rr$ can be analyzed via Lefschetz properties for $A_{\Theta}^\rr$. 
Our first result is the following.

\begin{theorem}
\label{almostSLP}
Let $A$ be a Cohen-Macaulay standard graded $K$-algebra of dimension $d$
and let $\Theta=\theta_1,\dots,\theta_d$ be an l.s.o.p.\ for $A$.
Let $r \geq 1$ be an integer and $s=\lfloor \frac {(r-1)d} r \rfloor$.
Then $A_\Theta^\rr$ has the $s$-Lefschetz property. 
Moreover, if $r \geq \deg h_A(t)$, then $A_\Theta^\rr$ is almost strong Lefschetz.
\end{theorem}

In commutative algebra,
the study of the weak Lefschetz property of Artinian graded $K$-algebras has shown to be of great interest.
Recall, that a standard graded Artinian $K$-algebra $A=\bigoplus_{i =0}^s A_i$ is said to have the {\em weak Lefschetz property}
if there is a linear form $w \in A_1$ such that the multiplication map $w: A_i \to A_{i+1}$
is either injective or surjective for all $i \geq 0$.
Also, we say that $A$ is {\em almost weak Lefschetz} if there is a $1 \leq p <s$
and a linear form $w \in A_1$ such that the multiplication map $w: A_i \to A_{i+1}$
is injective for $0\leq i \leq p-1$ and is surjective for $i \geq p+1$.
Note that we do not set any condition on the multiplication  map $w: A_p\rightarrow A_{p+1}$. 
If the multiplication map $w: A_p\rightarrow A_{p+1}$ is neither injective nor surjective, then the integer $p$ will be referred to as the {\em gap} of $A$ (w.r.t.\ $w$).
We obtain the following result for the almost weak and the weak Lefschetz property.

\begin{theorem}
\label{WLP}
Let $A$ be a Cohen-Macaulay standard graded $K$-algebra of dimension $d$ 
and let $\Theta=\theta_1,\dots,\theta_d$ be an l.s.o.p.\ for $A$.
\begin{itemize}
\item[(i)] If $r\geq \deg h_A(t)$, then $A_\Theta^\rr$ is almost weak Lefschetz.
\item[(ii)] If $d$ is even and $r\geq \max\{d,2 \deg h_A(t)-d\}$, then $A_\Theta^\rr$ has the weak
Lefschetz property.
\item[(iii)] If $d$ is odd, $r\geq \frac{d}{2}$ and $\deg h_A(t)\leq \frac{d}{2}$, then $A_\Theta^\rr$ has the 
weak Lefschetz property.
\end{itemize}
\end{theorem}

In \ref{sect:Lefschetz}, 
for the almost weak Lefschetz property, we will show a result which is somewhat stronger, showing in particular that for $d\leq \deg h_A(t)$ a weaker assumption on $r$ is sufficient for guaranteeing the almost weak Lefschetz property. 

We say that a polynomial 
$h_0+h_1t+\cdots+h_st^s \in \ZZ_{\geq 0}[t]$ is {\em unimodal} if 
there is a $1 \leq p \leq s$ such that $h_0\leq h_1 \leq \cdots\leq h_p \geq h_{p+1} \geq \cdots \geq h_s$.
Clearly,
if a standard graded Artinian $K$-algebra $A$ is almost weak Lefschetz, then the $h$-polynomial of $A$ is unimodal.
By using \ref{almostSLP} and \ref{WLP} together with a simple Gr\"obner basis argument,
we also prove the following result on $h$-polynomials.

\begin{theorem}
\label{main}
Let $A$ be a Cohen-Macaulay standard graded $K$-algebra of dimension $d$. 
Let $r \geq 1$ be an integer and $s=\lfloor {(r-1)d \over r} \rfloor$.
\begin{itemize}
\item[(i)] If $r \geq \frac 1 2 (\deg h_A(t) +1)$, then $h_{A^\rr} (t)$ is the $f$-polynomial of a flag simplicial complex.
\item[(ii)] If $r \geq \deg h_A(t)$, then $h_{A^\rr}(t)$ is unimodal and $g_{A^\rr}(t)$ is the $f$-polynomial of a simplicial complex.
\item[(iii)] Let $h_{A^\rr}(t)= \sum_{i \geq 0} h_i^\rr t^i$.
Then $h_i^\rr \leq h_{s-i}^\rr$ for all $i\leq \frac{s}{2}$.
\end{itemize}
\end{theorem}

\section{Lefschetz Properties}\label{sect:Lefschetz}

In this section,
we study Lefschetz properties of $A_\Theta^\rr$. In particular, we will provide the proofs of \ref{almostSLP} and \ref{WLP}. 

We start to fix some notation, which we will use throughout this section. In the following, we consider a Cohen-Macaulay standard graded $K$-algebra $A$ of dimension $d$
together with an l.s.o.p.\ $\Theta=\theta_1,\dots,\theta_d$ for $A$.
To prove \ref{almostSLP} and \ref{WLP},
we use the following observation, which relates the Hilbert series of $A_{\Theta}^\rr$ to the $h$-polynomial of the $r\textsuperscript{th}$ Veronese subalgebra of $A$:
By Cohen-Macaulayness of $A$, $\Theta$ is not only an l.s.o.p.\ but also a regular sequence for $A$. Hence, $A$ is a finitely generated and free $K[\theta_1,\ldots,\theta_d]$-module. In particular, there exist homogeneous elements $u_1,\ldots,u_m$ of $A$ such that 
we have the decomposition
\begin{equation}\label{eq:dec}
A=\bigoplus_{j=1}^m u_j \cdot K[\theta_1,\ldots,\theta_d]
\end{equation}
\noindent as $K[\theta_1,\ldots,\theta_d]$-modules.
Note that $u_1,\ldots, u_m$ is a $K$-basis of $A/\Theta A$ (see e.g., \cite[Chapter 1]{St-CCA}). 
Moreover, since the Hilbert series of $A/\Theta A$ is equal to the $h$-polynomial of $A$ (cf., \cite[Remark 4.1.11]{BH}),
we have
\begin{equation*}
\deg u_j \leq \deg h_A(t)
\end{equation*}
for all $1\leq j\leq m$. 
Let $r \geq 1$ be an integer. We will show that $\theta_1^{r},\ldots,\theta_d^r$ is an l.s.o.p. for $A^\rr$. 
From \eqref{eq:dec} we infer that the $r\textsuperscript{th}$ Veronese subalgebra $A^\rr$ decomposes as
\begin{equation*}
A^\rr=\bigoplus_{j=1}^m u_j \cdot \left(\bigoplus_{i \geq 0} K[\theta_1,\dots,\theta_d]_{ir - \deg u_j} \right),
\end{equation*}
where we set $K[\theta_1,\ldots,\theta_d]_k:=\{0\}$ if $k<0$.  
And for the quotient $A_{\Theta}^\rr=A^\rr/(\theta_1^r A^\rr + \cdots + \theta_d^r A^{\langle r\rangle})$ we obtain
\begin{equation}
\label{decomposition}
A_\Theta^\rr =\bigoplus_{j=1}^m u_j \cdot \left(\bigoplus_{i \geq 0} \big( K[\theta_1,\dots,\theta_d]/(\theta_1^r,\dots,\theta_d^r)\big)_{ir - \deg u_j} \right)
\end{equation}
as $K[\theta_1,\dots,\theta_d]^\rr$-modules.
Being the grading of $A^\rr_\Theta$ induced by the usual $\mathbb{Z}$-grading of $K[\theta_1,\ldots,\theta_d]$, we know that the homogeneous component of $A^\rr_\Theta$ of degree $i$ is given by
\begin{equation} \label{eq:degComp}
\bigoplus_{j=1}^m u_j \cdot \big( K[\theta_1,\dots,\theta_d]/(\theta_1^r,\dots,\theta_d^r)\big)_{ir - \deg u_j}.
\end{equation}
Since the right-hand side of \eqref{decomposition}
has finite length, we conclude that $\theta_1^r,\dots,\theta_d^r$ is an l.s.o.p.\ for $A^\rr$. This together with the fact that  the Cohen-Macaulay property is preserved under taking Veronese subalgebras (cf.\ \cite[Chapter 3]{GW}) implies that the Hilbert series of $A_\Theta^\rr$ equals the $h$-polynomial of $A^\rr$.
In particular, since $\max\{\deg u_j~|~1\leq j\leq m\}=\deg h_A(t)$
and since the maximum degree in $K[\theta_1,\ldots,\theta_d]/(\theta_1^r,\ldots,\theta_d^r)$ is $(r-1)d$,
it follows directly from \eqref{eq:degComp} that 
\begin{equation}\label{eq:maxDegree}
\deg h_{A^\rr}(t)=
\left\lfloor \frac{d(r-1)+\deg h_A(t)}{r}\right\rfloor.
\end{equation}
Note that the above equation \eqref{eq:maxDegree} holds for any standard graded $K$-algebra $A$ whose $h$-polynomial has non-negative coefficients
(e.g., use \cite[Theorem 1.2]{BW}).

For the proof of \ref{almostSLP} and \ref{WLP}, we need the following fact proved by Stanley \cite{St} and Watanabe \cite{Wa}.

\begin{lemma}
\label{SLPforCI}
Let $K$ be a field of characteristic $0$ and let $r \geq 1$ be an integer.
For integers $0 \leq i <j$,
the multiplication map
\begin{align*}
\times (x_1+\cdots+x_d)^{j-i}:
(K[x_1,\dots,x_d]/(x_1^r,\dots,x_d^r))_i 
\quad&\rightarrow\quad 
(K[x_1,\dots,x_d]/(x_1^r,\dots,x_d^r))_j\\
p\quad&\mapsto \quad(x_1+\cdots+x_d)^{j-i}\cdot p
\end{align*}
is injective if $i+j \leq (r-1)d$ and is surjective if $i+j \geq (r-1)d$.
\end{lemma}

We have now laid the necessary foundations for giving the proof of \ref{almostSLP}. \\

\noindent{\sf Proof of \ref{almostSLP}:}
Let $w=(\theta_1+\cdots+\theta_d)^r$.
We prove that $w$ is an $s$-Lefschetz element of $A_\Theta^\rr$,
namely, we will show that the multiplication
\begin{equation*}
\times w^{s-2i} : (A^{\langle r \rangle}_\Theta)_i \rightarrow (A^{\langle r \rangle}_\Theta)_{s-i}
\end{equation*}
is injective for $0\leq i\leq \lfloor \frac {s-1} 2 \rfloor.$
By the decomposition \eqref{decomposition},
it is enough to prove that, for $1\leq j\leq m$,
the multiplication
{\small
\begin{equation*}
\times w^{s-2i} :
\big( K[\theta_1,\dots,\theta_d]/(\theta_1^r,\dots,\theta_d^r)\big)_{ir - \deg u_j}
\rightarrow
\big( K[\theta_1,\dots,\theta_d]/(\theta_1^r,\dots,\theta_d^r)\big)_{(s-i)r - \deg u_j}
\end{equation*}
}is injective for the same $i$.
The desired injectivity follows from \ref{SLPforCI}
since $ir-\deg u_j + (s-i)r-\deg u_j \leq sr \leq (r-1)d$.

Finally,
if $r \geq \deg h_A(t)$, then $\deg \Hilb(A_\Theta^\rr,t)=\deg h_{A^\rr}(t) \leq s+1$ by \eqref{eq:maxDegree},
which implies that $A_\Theta^\rr$ is almost strong Lefschetz.
\qed
\smallskip

We now proceed to the proof of \ref{WLP}.
Part (i), i.\,e., the statement concerning the almost weak Lefschetz property, 
follows from the following stronger result.

\begin{theorem}\label{th:WLPexact}
Let $A$ be a Cohen-Macaulay standard graded $K$-algebra of dimension $d$ 
and let $\Theta=\theta_1,\dots,\theta_d$ be an l.s.o.p.\ for $A$. 
Then $A_\Theta^\rr$ is almost weak Lefschetz if
\begin{itemize}
\item[(a)] $d$ is even and one of the
following conditions holds:
\begin{itemize}
\item[(i)] $d \leq \frac{1}{2}\deg h_A(t)$ and $r \geq \frac{2 \deg h_A(t)-d}{3}$,
\item[(ii)] $\frac{1}{2} \deg h_A(t) \leq d \leq \deg h_A(t)$ and $r \geq d$,
\item[(iii)] $\deg h_A(t) \leq d \leq \frac{3}{2} \deg h_A(t)$ and $r \geq 2 \deg h_A(t) - d$,
\item[(iv)] $\frac{3}{2} \deg h_A(t) \leq d \leq 3 \deg h_A(t)$ and $r \geq \frac{d}{3}$,
\item[(v)] $d\geq 3 \deg h_A(t)$ and $r \geq \deg h_A(t)$, or, 
\end{itemize}
\item[(b)] $d$ is odd and one of the following conditions holds:
\begin{itemize}
\item[(i)] $d \leq \deg h_A(t)$ and $r \geq \deg h_A(t) -\frac{d}{2}$,
\item[(ii)] $\deg h_A(t) \leq d \leq 2 \deg h_A(t)$ and $r \geq \frac{d}{2}$,
\item[(iii)] $d\geq 2 \deg h_A(t)$ and $r \geq \deg h_A(t)$.
\end{itemize}
\end{itemize}
\end{theorem}

\begin{proof}
Before providing the proofs for each set of conditions separately, we start with a general discussion that can be used in all cases. 
Let $w:=(\theta_1+ \cdots+\theta_d)^r$. Our aim is to show that in all parts of the theorem, $w$ can be used as an almost weak Lefschetz element for $A_\Theta^\rr$. 
Using the same notations as at the beginning of this section, we know from \eqref{decomposition} that as $K[\theta_1,\ldots,\theta_d]^\rr$-modules, we have the decomposition 
\begin{equation*}
A^{\langle r \rangle}_\Theta=\bigoplus_{j=1}^m u_j \left(\bigoplus_{i \geq 0} \big( K[\theta_1,\dots,\theta_d]/(\theta_1^r,\dots,\theta_d^r)\big)_{ir - \deg u_j} \right).
\end{equation*}
Thus, in order to show that the multiplication
\begin{equation*}
\times w : (A^{\langle r \rangle}_\Theta)_i \rightarrow (A^{\langle r \rangle}_\Theta)_{i+1}
\end{equation*}
is injective and surjective for a certain $i\geq 0$, 
it suffices to show that
for all $1\leq j\leq m$ the multiplication
{\small \begin{equation}\label{eq:mult}
\times w :
\big( K[\theta_1,\dots,\theta_d]/(\theta_1^r,\dots,\theta_d^r)\big)_{ir - \deg u_j}
\rightarrow
\big( K[\theta_1,\dots,\theta_d]/(\theta_1^r,\dots,\theta_d^r)\big)_{(i+1)r - \deg u_j}
\end{equation}}
is injective and surjective, respectively for the same $i$.

We first consider case (a) (i). 
Suppose that $d$ is even,
$d \leq \frac{1}{2}\deg h_A(t)$ and $r \geq \frac{2 \deg h_A(t)-d}{3}$.
Combining the latter two conditions in particular yields $r\geq d$. 
Our aim is to use \ref{SLPforCI}. 
We first show that the multiplication in \eqref{eq:mult} is injective for $0\leq i\leq \frac{d}{2}-1$ and for all $1\leq j\leq m$. 
For all $1\leq j\leq m$ it holds that
\begin{equation*}
2ir+r-2\deg u_j \leq dr-r \leq (r-1)d + d-r \leq (r-1)d,
\end{equation*}
where the first and the last inequality follow from $\deg u_j\geq 0$ for $1\leq j\leq m$ and $r \geq d$, respectively. 
Hence, \ref{SLPforCI} implies the desired injectivity.

Next, we show that the multiplication in \eqref{eq:mult} is surjective for $i\geq \frac{d}{2}+1$. As in the previous case, for $1\leq j\leq m$ we compute
\begin{equation*}
2ir + r -2 \deg u_j \geq dr + 3r -2 \deg h_A(t) \geq (r-1)d,
\end{equation*}
where for the first inequality we use that $\deg u_j \leq \deg h_A(t)$ for $1\leq j\leq m$, and the last inequality holds since $r \geq  \frac{2 \deg h_A(t)-d}{3}$. 
Surjectivity now follows from \ref{SLPforCI}.

The cases (a) (ii) -- (iv) and (b) (i) -- (ii) follow from almost literally the same arguments, taking into account the different ranges and bounds for $d$ and $r$, respectively, as well as the different location of the gap. 
Indeed, if there is a gap, then it is at position $\frac{d}{2}$ in the cases (a) (i) -- (ii), 
and at position $\frac{d}{2}-1$ in the cases (iii) -- (iv).
In the situation of (b) (i) -- (ii), 
the gap --~if existing~-- lies at position $\frac{d-1}{2}$.

The cases (a) (v) and (b) (iii) have to be treated slightly differently. Let $s=\lfloor \frac{(r-1)d}{r}\rfloor$. 
By an analog reasoning as for the other cases one infers that the multiplication in \eqref{eq:mult} is surjective for $ i \geq \frac s 2 +1$.
On the other hand,
\ref{almostSLP} says that $A_\Theta^\rr$ is $s$-Lefschetz.
In particular the multiplication map in \eqref{eq:mult} is injective for $i \leq \frac{s-1}{2}$. Hence, we conclude that $A_\Theta^\rr$ is almost weak Lefschetz with a possible gap at position $\lfloor \frac{s+1}{2} \rfloor$. 
\end{proof}

We want to remark that the arguments in the above proof do only depend on the effective size of $r$ and not on the precise relation between $d$ and $\deg h_A(t)$. Moreover, the proofs of (a) (iv) and (b) (iii) 
do not use the fact that $d\leq 3\deg h_A(t)$ and $d\leq 2\deg h_A(t)$, respectively. We only include these restrictions since for $d> \deg h_A(t)$ part (a) (v) and part (b) (iii) provide better, i.\,e., smaller  bounds for $r$. In particular, this allows us to 
conclude, that if $d$ is even, the gap --~if existing~-- is at position $\frac{d}{2}$ if $r\geq \max\{d,\frac{2\deg h_A(t)-d}{3}\} $ and at position $\frac{d}{2}+1$ if $r\geq \max\{\frac{d}{3},2\deg h_A(t)-d\}$. If $d$ is odd and 
$r\geq \max\{\frac{d}{2},\deg h_A(t)-\frac{d}{2}\}$, the gap is at position $\frac{d-1}{2}$. 
This will be relevant for the proof of \ref{WLP} (ii) and (iii). \\
\smallskip

\noindent{\sf Proof of \ref{WLP}:} 
Part (i) can readily be deduced from \ref{th:WLPexact}.
To show part (ii), note that --~independent of $d$~-- it follows from \ref{th:WLPexact} (a) (i) --(iv) and the discussion preceding this proof that $A_\Theta^\rr$ is almost weak Lefschetz. Since there can exist 
at most one gap, we infer from the mentioned discussion that $A_\Theta^\rr$ is indeed weak Lefschetz.

For part (iii), let $r\geq \max\{\frac{d}{2},\deg h_A(t)-\frac{d}{2}\}$. By \ref{th:WLPexact} and the discussion preceding this proof, we know that $A_\Theta^\rr$ is almost weak Lefschetz with 
a possible gap being at position $\frac{d-1}{2}$. 
Assume, in addition, that $\deg h_A(t)\leq \frac{d}{2}$. We claim that the multiplication map
\begin{equation*}
 \times w: (A_\Theta^\rr)_{\frac{d-1}{2}}\rightarrow (A_\Theta^\rr)_{\frac{d-1}{2}+1}
\end{equation*}
is surjective.
The desired surjectivity follows from 
\ref{SLPforCI} since for all $1\leq j\leq m$ it holds that 
\begin{equation*}
r\left(\frac{d-1}{2}+\frac{d-1}{2}+1\right)-2\deg u_j \geq 
rd-2\deg h_A(t) \geq (r-1)d.
\end{equation*}
We conclude that $A_\Theta^\rr$ has the weak Lefschetz property.
\qed

\begin{remark}
\ref{WLP}(ii) says that, for any even dimensional Cohen-Macaulay graded $K$-algebra $A$,
the algebra $A_\Theta^\rr$ has the weak Lefschetz property for $r \gg 0$.
Unfortunately, this fact does not hold for odd dimensional Cohen-Macaulay graded $K$-algebras.
Let $A=K[x_1,\dots,x_8]/((x_1^2,x_1x_2,x_1x_3,x_1x_4,x_1x_5)+(x_2,x_3,x_4,x_5)^3)$.
Then $A$ is a Cohen-Macaulay graded $K$-algebra of dimension $3$ with the $h$-polynomial $h_A(t)=1+5t+10t^2$
and $\Theta=x_6,x_7,x_8$ is an l.s.o.p.\  for $A$,
but $A_\Theta^\rr$ does not have the weak Lefschetz property for any $ r \geq 3$.

If $r \geq 3$, then the $h$-polynomial $h_{A^\rr}(t)=h^\rr_0+h^\rr_1 t+h^\rr_2 t^2$
of $A^\rr$ satisfies $h_0^\rr<h_1^\rr<h_2^\rr$.
However,
there are no linear forms $w$ such that
$\times w: (A_\Theta^\rr)_1 \to (A_\Theta^\rr)_2$ is injective.
Consider the $K$-vector spaces 
$V=x_1 (K[x_6,x_7,x_8]/(x_6^r,x_7^r,x_8^r))_{r-1} \subset (A_\Theta^\rr)_1$
and
$W=x_1 (K[x_6,x_7,x_8]/(x_6^r,x_7^r,x_8^r))_{2r-1} \subset (A_\Theta^\rr)_2$.
Then,  for any linear form $w\in A_\Theta^\rr$ we have $w V \subset W$, since $x_1 x_i=0$ in $A$ for $i=1,2,\dots,5$, but 
\begin{eqnarray*}
\dim_K V &=&\dim_K (K[x_6,x_7,x_8]/(x_6^r,x_7^r,x_8^r))_{r-1}\\
&>&\dim_K(K[x_6,x_7,x_8]/(x_6^r,x_7^r,x_8^r))_{2r-1} = \dim_K W,
\end{eqnarray*}
where the inequality follows since
$\dim_K (K[x_6,x_7,x_8]/(x_6^r,x_7^r,x_8^r))_{r-1}= { r+1 \choose r-1}$
and
$\dim_K (K[x_6,x_7,x_8]/(x_6^r,x_7^r,x_8^r))_{2r-1}=
\dim_K (K[x_6,x_7,x_8]/(x_6^r,x_7^r,x_8^r))_{r-2}={ r \choose r-2}$.
This fact implies that the multiplication $\times w: V \to W$ is not injective.
\end{remark}

\section{Consequences on $h$-vectors}

In this section,
we prove \ref{main}. 
Throughout this section, we let $S=K[x_1,\dots,x_n]$ be a standard graded polynomial ring over a field $K$.
For an integer $r \geq 1$,
let
$$T_{\rr} =K[z_m: m \mbox{ is a monomial in $S$ of degree $r$]},$$
where each $z_m$ is a variable.
Then there is a natural ring homomorphism
\begin{align*}
\begin{array}{lclc}
\phi_r: &\ T_\rr &\longrightarrow& S^\rr\\
& z_m &\mapsto& m.
\end{array}
\end{align*}
For a homogeneous ideal $I \subset S$,
let $I^\rr:=\bigoplus_{j \geq 0} I_{jr}$.
Then $I^\rr$ is a graded ideal of $S^\rr$ and $(S/I)^\rr = S^\rr / I^\rr$.
Also, the ring $(S/I)^\rr$ is isomorphic to $T_\rr/ \phi_r^{-1}(I^\rr)$.

To prove the main result,
we need a few known results on Gr\"obner bases of $\phi^{-1}(I^\rr)$ proved by Eisenbud, Reeves and Totaro \cite{ERT}.
We refer the readers to \cite{CLO} for the basics on Gr\"obner basis theory.

Let $>_\rev$ be the reverse lexicographic order on $S$ induced by $x_1>\cdots >x_n$,
and let $\succ_\rev$ be the reverse lexicographic order on $T_\rr$ such that the ordering of the variables is defined by
$z_m \succ_\rev z_{m'}$ if $m >_\rev m'$.
For a monomial $m \in S$,
we write $\max(m)$ (resp.\ $\min(m)$)
for the largest (resp.\ smallest) integer $i$ such that $x_i$ divides $m$.
We say that a monomial
$$u=z_{m_1} z_{m_2} \cdots z_{m_k} \in T_\rr,$$
where $m_1 >_\rev \cdots >_\rev m_k$,
is {\em standard} if $\max(m_i) \leq \min (m_{i+1})$ for $1\leq i\leq k-1$.
The following fact was shown in the proof of \cite[Proposition 6]{ERT}.

\begin{lemma}
\label{3-1}
A monomial $u \in T_\rr$ is standard if and only if $u \not \in \init_{\succ_\rev} (\ker \phi_r)$.
\end{lemma}

The above lemma implies the next result.

\begin{lemma}
\label{3-2}
Let $ r \geq 1$ and $ 1 \leq \ell \leq n$ be integers.
Let $I \subset S$ be a monomial ideal
and $J=I+(x_n^r,x_{n-1}^r,\dots,x_\ell^r)$.
Then
$$
\init_{\succ_\rev} \phi_r^{-1} (J^\rr)
= \init_{\succ_\rev} \phi_r^{-1}(I^\rr) +(z_{x_n^r},\dots,z_{x_\ell^r})
+(z_mz_m': mm' \in (x_n^r,\dots,x_\ell^r)).$$
\end{lemma}

\begin{proof}
It is clear that the left-hand side contains the right-hand side.
We show that also the reverse inclusion holds. 
Let
$$u=z_{m_1} z_{m_2} \cdots z_{m_k} \in \init_{\succ_\rev} \phi_r^{-1}(J^\rr),$$
be a monomial with $m_1 >_\rev \cdots >_\rev m_k$.
We show that if $u \not \in \init_{\succ_\rev} \phi_r^{-1}(I^\rr)$, 
then $u  \in (z_{x_n^r},\dots,z_{x_\ell^r})+(z_mz_{m'}:mm' \in (x_n^r,\dots,x_\ell^r))$.

Since $u \not \in \init_{\succ_\rev} \phi_r^{-1}(I^\rr)$,
we have $u \not \in \init_{\succ_\rev} \ker (\phi_r)$.
Thus $u$ is standard by \ref{3-1}.
We claim $\phi_r(u) \in J^\rr$.
Let $f=u+v_1+\cdots+v_m+g \in \phi_r^{-1}(J^\rr)$
be such that $\init_{\succ_\rev}(f)=u$, $g \in \ker (\phi_r)$ and $u,v_1,\dots,v_m$ are distinct standard monomials.
Then 
$\phi_r(f)=\phi_r(u)+\phi_r(v_1)+\cdots+\phi_r(v_m) \in J^\rr.$
Since $J$ is a monomial ideal and $\phi_r(u),\phi_r(v_1),\dots,\phi_r(v_m)$ are distinct monomials,
we have $\phi_r(u) \in J^\rr$.

Since, by assumption, $u \not \in \phi_r^{-1}(I^\rr)$,  
we have
$$\phi_r(u)=m_1m_2\cdots m_k \in (x_n^r,\dots,x_\ell^r).$$
Thus, there is an $ \ell \leq i \leq n$ such that $x_i^r$ divides $\phi_r(u)$.
If $\deg u =1$, then $u$ must be equal to $z_{x_i^r}$.
If $\deg u >1$,
then, by the definition of a standard monomial,
there is a $1 \leq j \leq k-1$ such that $x_i^r$ divides $m_jm_{j+1}$.
This proves the desired statement.
\end{proof}

A monomial ideal $I \subset S$ is called {\em stable} if,
for any monomial $m \in I$,
one has $m ( x_i / x_{\max(m)}) \in I$ for any $i < \max(m)$.
The following facts are known.

\begin{lemma}\
\label{3-3}
\begin{itemize}
\item[(i)] For any Cohen-Macaulay standard graded $K$-algebra $A$ with $\dim_K A_1 \leq n$,
there is a stable monomial ideal $J \subset S$ such that $S/J$ is Cohen-Macaulay and $S/J$ has the same Hilbert series as $A$.
\item[(ii)] Let $I \subset S$ be a stable monomial ideal such that $S/I$ is a Cohen-Macaulay graded $K$-algebra of dimension $d$.
Then $x_n,x_{n-1},\dots,x_{n-d+1}$ is a linear system of parameters for $S/I$
and $I$ is generated by monomials of degree $\leq \deg h_{S/I}(t) +1$.
\end{itemize}
\end{lemma}

\begin{proof}
We only sketch the proof since the statements are well-known in commutative algebra.
For any standard graded $K$-algebra $A$ with $\dim_K A_1 \leq n$,
there is a homogeneous ideal $ I \subset S$ such that $S/I$ is isomorphic to $A$ as graded $K$-algebra.
Then (i) follows from \cite[Theorem 2]{IP}.

Suppose that $I$ is a stable monomial ideal such that $S/I$ is Cohen-Macaulay.
A result of Eliahou and Kervaire \cite{EK} shows that $I$ is generated by monomials in $K[x_1,\dots,x_{n-d}]$.
This shows that $x_n,x_{n-1},\dots,x_{n-d+1}$ is a regular sequence of $S/I$
and, therefore, is an l.s.o.p.\ for $S/I$.
Also,
since the $h$-polynomial of $S/I$ is equal to
the Hilbert series of $S/(I+(x_n,x_{n-1},\dots,x_{n-d+1}))$,
$I$ contains all monomials in $K[x_1,\dots,x_{n-d}]$ of degree $\deg h_{S/I}(t)+1$.
Since $I$ is generated by monomials in $K[x_1,\dots,x_{n-d}]$,
$I$ is generated by monomials of degree $\leq \deg h_{S/I}(t) +1$.
\end{proof}

For the proof of \ref{main} we will use the following result
for Veronese algebras of the quotient of a stable monomial ideal, which was proven by 
Eisenbud, Reeves and Totaro \cite[Theorem 8]{ERT}.

\begin{lemma}[Eisenbud-Reeves-Totaro]
\label{ert}
Let $I \subset S$ be a stable monomial ideal generated by monomials of degree $\leq s$.
If $r \geq {s \over 2}$, then $\init_{\succ_\rev} \phi^{-1}_r (I^{\langle r \rangle})$ is generated by monomials of degree $\leq 2$.
\end{lemma}

Now we are in the position to prove \ref{main}.
Recall that a {\em simplicial complex} $\Delta$ on $[n]:=\{1,2,\dots,n\}$
is a collection of subsets of $[n]$ satisfying that if $F \in \Delta$ and $G \subset F$, then $G \in \Delta$.
A simplicial complex is said to be {\em flag} if every minimal non-face of $\Delta$ has cardinality $\leq 2$.
For a simplicial complex $\Delta$,
we write $f_i(\Delta)$ for the number of elements $F \in \Delta$ with $|F|=i+1$.
The {\em $f$-polynomial} of $\Delta$ is the polynomial
$f(\Delta,t)= \sum_{i\geq 0} f_{i-1}(\Delta) t^i$,
where $f_{-1}(\Delta):=1$.
The $f$-polynomial of $\Delta$ can be expressed in an algebraic way.
Indeed,
the $f$-polynomial of a simplicial complex $\Delta$ on $[n]$
is equal to the Hilbert series of $S/((x_F: F \not \in \Delta)+(x_1^2,\dots,x_n^2))$,
where $x_F:=\prod_{i \in F} x_i$.
Moreover, $\Delta$ is flag if and only if the ideal $(x_F: F \not \in \Delta)+(x_1^2,\dots,x_n^2)$ is generated by monomials of degree $\leq 2$.
\bigskip

\noindent{\sf Proof of \ref{main}:} 
Part (iii) immediately follows from \ref{almostSLP}.
The unimodality of (ii) is a direct consequence of \ref{WLP}.
We prove (i) and the remaining part of (ii).

Fix $r \geq 1$.
Since the Hilbert series of $A^\rr$ only depends on $r$
and the Hilbert series of $A$,
by \ref{3-3} (i),
we may assume that $A=S/I$, where $I$ is a stable monomial ideal.
Let $\Theta=x_n,x_{n-1},\dots,x_{n-d+1}$
and $J=I+(x_n^r,\dots,x_{n-d+1}^r)$.
Then, by \ref{3-3} (ii), $\Theta$ is an l.s.o.p.\ for $A=S/I$
and
$$A_\Theta^\rr= S^\rr/J^\rr \cong T_\rr / \phi_r^{-1}(J^\rr).$$

We now prove (i).
Suppose $r \geq \frac 1 2 (\deg h_A(t) +1)$.
Let $\Delta$ be the set of monomials in $T_\rr$, 
which are not contained in $\init_{\succ_\rev} (\phi_r^{-1}(J^\rr))$.
By \ref{3-3} (ii),
$I$ is generated by monomials of degree $\leq \deg h_A(t)+1$.
Then \ref{3-2} and \ref{ert} say that
$\init_{\succ_\rev} (\phi_r^{-1}(J^\rr))$ is generated by monomials of degree $\leq 2$.
This fact shows that
$\init_{\succ_\rev} (\phi_r^{-1}(J^\rr))$ contains $z_m^2$ for any variable $z_m$ of $T_\rr$
since 
$T_\rr/\phi_r^{-1}(J^\rr)$ is Artinian. 
This implies that $\Delta$ is a set of squarefree monomials.
Thus, we may regard $\Delta$ as a simplicial complex.
Moreover,
since 
$$\init_{\succ_\rev} (\phi_r^{-1}(J^\rr))=(u: u \mbox{ is a monomial in $T_\rr$ with $u\not \in \Delta$})$$
 is generated by monomials of degree $\leq 2$,
$\Delta$ is a flag simplicial complex.
Also, by the construction of $\Delta$,
we have
$$
f(\Delta,t)= \Hilb(T_\rr/\phi_r^{-1}(J^\rr),t)
=\Hilb(A_\Theta^\rr,t)=h_{A^\rr}(t),$$
which proves (i).

Finally, we prove the second part of (ii).
Suppose $r \geq \deg h_A(t)$.
Let $\lambda=\deg h_{A^\rr}(t)$.
By \ref{WLP} and the proof of \ref{th:WLPexact},
there is a linear form $w \in (S^\rr)_1 =S_r$ such that
\begin{align}
\label{3.a}
g_{A^\rr}(t)=
\sum_{i=0}^{\lfloor \frac \lambda 2\rfloor}
\big(\dim_K (A_\Theta^\rr/w A_\Theta^\rr)_i\big) t^i.
\end{align}
Observe
\begin{align}
\label{3.b}
A_\Theta^\rr/w A_\Theta^\rr \cong T_\rr/ \phi^{-1}_r (J^\rr+(w)^\rr).
\end{align}
Let
$\Gamma$ be the set of monomials of degree $\leq \lfloor \frac \lambda 2 \rfloor$
which are not in $\init_{\succ_\rev}\phi^{-1}_r (J^\rr+(w)^\rr)$.
As we have already seen in the proof of (i),
$\init_{\succ_\rev}\phi^{-1}_r (J^\rr)$ contains $z_m^2$
for any variable $z_m$ of $T_\rr$.
Thus $\Gamma$ can be regarded as a simplicial complex.
Then, \eqref{3.a} and \eqref{3.b} say that
$g_{A^\rr}(t)$ is equal to the $f$-polynomial of $\Gamma$, as desired.
\qed


  \bibliography{biblio}

\begin{thebibliography}{10}

\bibitem{BS}
M.\ Beck and A.\ Stapledon.
\newblock On the log-concavity of {H}ilbert series of {V}eronese subrings and
  {E}hrhart series.
\newblock {\em Math. Z.}, 264:195--2007, 2010.

\bibitem{BW}
F.\ Brenti and V.\ Welker.
\newblock The {V}eronese construction for formal power series and graded
  algebras.
\newblock {\em Adv. Appl. Math.}, 42:545--556, 2009.

\bibitem{BH}
W.\ Bruns and J.\ Herzog.
\newblock {\em Cohen-{M}acaulay rings. {R}ev. ed.}, volume~39 of {\em Cambridge
  Studies in Advanced Mathematics}.
\newblock Cambridge University Press, 1998.

\bibitem{CLO}
D.\ Cox, J.\ Little, and D.\ O'Shea.
\newblock {\em Ideals, {V}arieties, and {A}lgorithms: {A}n {I}ntroduction to
  {C}omputational {A}lgebraic {G}eometry and {C}ommutative {A}lgebra}.
\newblock Springer-Verlag, 1997.

\bibitem{ERT}
D.\ Eisenbud, A.\ Reeves, and B.\ Totaro.
\newblock Initial ideals, {V}eronese subrings, and rates of algebras.
\newblock {\em Adv. Math.}, 109(2):168--187, 1994.

\bibitem{EK}
S.\ Eliahou and M.\ Kervaire.
\newblock Minimal resolutions of some monomial ideals.
\newblock {\em J. Algebra}, 129:1--25, 1990.

\bibitem{GW}
S.\ Goto and K.\ Watanabe.
\newblock On graded rings. i.
\newblock {\em J. Math. Soc. Japan}, 30:179--213, 1978.

\bibitem{IP}
S.\ Iyenger and K.\ Pardue.
\newblock Maximal minimal resolutions.
\newblock {\em J.\ reine angew.\ Math.}, 512:27--46, 1999.

\bibitem{KN}
M.\ Kubitzke and E.\ Nevo.
\newblock The {L}efschetz property for barycentric subdivisions of shellable
  complexes.
\newblock {\em Trans. Am. Math. Soc.}, 361(11):6151--6163, 2009.

\bibitem{KW}
M.\ Kubitzke and V.\ Welker.
\newblock Enumerative $g$-theorems for the {V}eronese construction for formal
  power series and graded algebras.
\newblock {\em http://arxiv.org/abs/1108.2852}, 2011.

\bibitem{St}
R.P.\ Stanley.
\newblock Weyl groups, the hard {L}efschetz theorem, and the {S}perner
  property.
\newblock {\em SIAM J. Algebraic Discrete Methods}, 1:168--184, 1980.

\bibitem{St-CCA}
R.P.\ Stanley.
\newblock {\em Combinatorics and {C}ommutative {A}lgebra}, volume~41 of {\em
  Progress in Mathematics}.
\newblock Birkh\"auser Boston Inc., Boston, MA, second edition, 1996.

\bibitem{Wa}
J.\ Watanabe.
\newblock The {D}ilworth number of {A}rtinian rings and finite posets with rank
  function.
\newblock In {\em Commutative algebra and combinatorics ({K}yoto, 1985)},
  volume~11 of {\em Adv. Stud. Pure Math.}, pages 303--312. North-Holland,
  Amsterdam, 1987.

\end{thebibliography}
  \bibliographystyle{plain}

\end{document}